\documentclass[12pt]{amsart}

\newcommand{\hide}[1]{}

\usepackage{amssymb}
\usepackage{amsbsy}
\usepackage{amscd}
\usepackage[mathscr]{eucal}
\usepackage{verbatim}
\usepackage{version}
\usepackage{color}
\usepackage[matrix,arrow]{xy}
\usepackage{geometry}

\numberwithin{equation}{section}

\theoremstyle{plain}
\newtheorem{thm}{Theorem}[section]

\newtheorem{prop}[thm]{Proposition}

\newtheorem{conj}[thm]{Conjecture}
\newtheorem{cor}[thm]{Corollary}
\newtheorem{lem}[thm]{Lemma}

\theoremstyle{definition}
\newtheorem{defi}[thm]{Definition}

\theoremstyle{remark}

\newtheorem{rem}[thm]{Remark}

\newcommand{\BA}{{\mathcal B}{\mathcal A}}

\newcommand{\C}{{\mathcal C}}

\newcommand{\MV}{{\mathcal M}{\mathcal V}}

\newcommand{\X}{{\mathcal X}}

\newcommand{\RealNumbers}{{\mathbb R}}
\newcommand{\Integers}{{\mathbb Z}}
\newcommand{\ComplexNumbers}{{\mathbb C}}
\newcommand{\RationalNumbers}{{\mathbb Q}}

\newcommand{\linsys}[1]{{\mid}#1{\mid}}

\newcommand{\NE}{{\rm NE}}
\newcommand{\Bir}{{\rm Bir}}

\def\cal{\mathcal}
\def\Bbb{\mathbb}

\newenvironment{NB}{
\color{red}{\bf NB}. \footnotesize 
}{}
\excludeversion{NB}

\newcommand{\Aut}{\operatorname{Aut}}

\newcommand{\Amp}{\operatorname{Amp}}

\newcommand{\Nef}{\operatorname{Nef}}
\newcommand{\Eff}{\operatorname{Eff}}

\begin{document}

\title[]
{A proof of the Kawamata-Morrison Cone Conjecture for 
holomorphic symplectic varieties 
of $K3^{[n]}$ or generalized Kummer deformation type}
\author{Eyal Markman and Kota Yoshioka}
\address{Department of Mathematics and Statistics, 
University of Massachusetts, Amherst, MA 01003, USA}
\email{markman@math.umass.edu}
\address{Department of Mathematics, Faculty of Science, 
Kobe University, Kobe, 657, Japan }
\email{yoshioka@math.kobe-u.ac.jp}
\date{\today}

\begin{abstract}
We prove a version of the Kawamata-Morrison ample cone conjecture for 
projective irreducible holomorphic symplectic manifolds 
deformation equivalent to 
either the Hilbert scheme of $n$ points on a $K3$ surface, or
a generalized Kummer variety.
\end{abstract}

\maketitle

\tableofcontents

%
\section{Introduction}

An {\em irreducible holomorphic symplectic manifold} is a 
simply connected compact K\"{a}hler manifold $X$, 
such that $H^0(X,\wedge^{2}T^*X)$ is one-dimensional, 
spanned by an everywhere non-degenerate holomorphic $2$-form.
There is a natural integral non-degenerate  symmetric bilinear pairing on $H^2(X,\Integers)$, 
of signature $(3,b_2(X)-3)$, called the
Beauville-Bogomolov-Fujiki pairing \cite{beauville}. 
It is positive on K\"{a}hler classes and indivisible
in the sense that $\gcd\{(x,y) \ : \ x,y\in H^2(X,\Integers)\}=1$. 
Let $X$ be a projective irreducible holomorphic symplectic manifold. 
Set $\Lambda:=H^{1,1}(X,\Integers)$. 
Then $\Lambda$ is a lattice of signature $(1,\rho-1),$ for a positive integer $\rho$. 
In particular, the rational vector space $\Lambda_\RationalNumbers$ is self-dual with respect to the
restriction of the Beauville-Bogomolov-Fujiki pairing. 

Let $\C\subset \Lambda_\RealNumbers$ be the positive cone. $\C$ is the connected component of the cone
$\{\lambda\in \Lambda_\RealNumbers \ : (\lambda,\lambda)>0\}$, which contains the ample cone of $X$. 
Denote the ample cone of $X$ by $\Amp_X\subset \C$. Let $\Nef_X$ be the closure of
$\Amp_X$ in $\Lambda_\RealNumbers$. 
Let $\Nef_X^*\subset \Lambda_\RealNumbers$ be the dual cone. 
A non-zero class $\lambda$ of $\Nef_X^*$ is {\em extremal} 
if every decomposition $\lambda=\lambda_1+\lambda_2$,
with $\lambda_i\in \Nef_X^*$,  consists of $\lambda_i$ which are scalar multiples of $\lambda$. 
The Beauville-Bogomolov-Fujiki pairing identifies $H^2(X,\RationalNumbers)$ with $H_2(X,\RationalNumbers)$
and the dual cone $\Nef_X^*$ gets identified with the Mori cone $\overline{\NE}_1(X)$. Under this identification, 
extremal classes in the above sense generate extremal rays of the Mori cone.

\begin{conj}
\label{conj-self-intersection-of-extremal-classes-bounded-below}
Let $Y$ be an irreducible holomorphic symplectic manifold birational to $X$.
The set $\{(e,e) \ : \ 
e \ \mbox{is an integral, primitive, and extremal class in} \ \Nef_Y^*
\}$
is bounded below
by a constant, which depends only on the birational class of $X$.
\end{conj}

We say that $X$ is of {\em $K3^{[n]}$-type}, if $X$ is deformation equivalent to the
Hilbert scheme $S^{[n]}$ of length $n$ subschemes of a $K3$ surface $S$.
The conjecture is known to hold if $X$ is of  $K3^{[n]}$-type, by independent results of
Mongardi \cite[Theorem 1.3]{mongardi} and Bayer-Hassett-Tschinkel \cite{BHT}.
These two papers rely on the proof of the analogous result for moduli spaces of sheaves 
on $K3$ surfaces, by Bayer and Macri \cite[Theorem 12.1]{bayer-macri-mmp}. 
The conjecture is also known when $X$ is 
of generalized Kummer type, that is,
$X$ is deformation equivalent to one of the generalized Kummer
varieties associated to an abelian surface, by the work of the second author \cite{yoshioka-ample-cone}
and the results of Mongardi \cite{mongardi} and Bayer-Hassett-Tschinkel \cite{BHT}.

Let $\Nef_X^+$ be the convex hull of 
$\Nef_X\cap\Lambda_\RationalNumbers$
in $\Lambda_\RealNumbers$. We have the inclusions
\[
\Amp_X\subset \Nef_X^+ \subset \Nef_X.
\]

\begin{defi}
A {\em rational polyhedral cone} in $ \Lambda_\RealNumbers$ is a closed convex cone spanned by a finite set of 
(integral) classes of $\Lambda$. 
\end{defi}

The main result of this note is the following version of the Kawamata-Morrison ample cone conjecture.
\begin{thm} 
[Theorem \ref{thm-ample-cone-conj} below]
\label{thm-introduction}
Assume that Conjecture \ref{conj-self-intersection-of-extremal-classes-bounded-below} holds for $X$. Then 
there exists a rational polyhedral cone $D\subset \Nef_X^+$, which is a fundamental domain for the 
action of the automorphism group of $X$ on  $\Nef_X^+$.
\end{thm}

\begin{rem}
Let $\Eff_X$ be the cone generated by effective divisor classes.
The Kawamata-Morrison cone conjecture is often stated in terms of the cone
$\Nef_X^e:=\Nef_X\cap \Eff_X$ instead of the cone $\Nef_X^+$ 
\cite{kawamata,morrison-compactifications,totaro}.
The inclusion $\Nef_X^e\subset \Nef_X^+$ follows from Bouksom's divisorial Zariski decomposition
for all irreducible holomorphic symplectic manifolds
\cite[Theorem 4.3]{boucksom-zariski-decomposition}. The equality $\Nef_X^e= \Nef_X^+$ is known
when $X$ is of $K3^{[n]}$-type and follows from 
the statement that integral isotropic nef classes are effective
\cite[Cor. 1.6]{markman-lagrangian}. 
 \end{rem}

\hide{
{\it  
The following is just for Kota's understanding:
Let $x$ be a class in the effective cone.
Then there are integral and effective classes $x_1,...,x_n$ such that
$x=\sum_{i=1}^n t_i x_i$, $t_i \in {\Bbb R}_{>0}$. 
By the Zariski decomposition, 
we may assume that $x_1,...,x_k$ are prime exceptional divisors 
and $x_{k+1},...,x_n$ belongs to the closure of positive cone.
We set $y_1:=\sum_{i=1}^k t_i x_i$ and $y_2:=\sum_{i=k+1}^n t_i x_i$.
If $x \in \Nef_X$, then $(x,y_1) \geq 0$ by the characterization of
$\overline{\MV}_X$.
On the other hand $(y_1^2) \leq 0$ and 
the equality implies $y_1=0$ by the .
Hence $(y_2,y_1) >0$ or $y_1=0$.
Then $(x^2)>0$ or $x=y_2$.
For the latter case, $(y_2^2)>0$ or 
$(y_2^2)=0$ and $y_2=t_i x_i$ for some $i$.

Assume that $(x^2)>0$.
We take a rational polyhedral cone $\Delta$
in ${\cal C}$ containing
$x$. 
Since $\Nef_X$ is defined by finitely many walls in $\Delta$,
in the rational polyhedral cone $(\sum_i {\Bbb R}_{\geq 0} x_i)$,
$\Delta \cap \Nef_X 
\cap (\sum_i {\Bbb R}_{\geq 0} x_i)$ is rational polyhedral.
Hence $x=\sum_j t_i' x_i'$, $t_i' \in {\Bbb R}_{>0}$ and
$x_i'$ are rational point of $\Nef_X \cap {\cal C}$.
Thus $x \in \Nef_X^+$. 
}}
A related result is the following.

\begin{cor} [Corollary \ref{cor-finitely-many-isomorphism-classes} below]
Assume that Conjecture \ref{conj-self-intersection-of-extremal-classes-bounded-below} holds for $X$. Then 
the set of isomorphism classes of irreducible holomorphic symplectic manifolds 
in the birational class of $X$ is finite.
\end{cor}

A version of Theorem \ref{thm-introduction} was proven independently by 
Amerik and Verbitsky for irreducible holomorphic symplectic manifolds, not necessarily projective
\cite{amerik-verbitsky}.

{\bf Acknowledgements:} We thank Artie Prendergast-Smith for pointing out to us that the main
Theorem \ref{thm-ample-cone-conj} follows from Corollary \ref{cor-nef-cone-thm}.
We thank Eduard Looijenga for the helpful reference to his work \cite{looijenga}. 
The work of Eyal Markman was partially supported by a grant  from the Simons Foundation (\#245840), and by NSA grant H98230-13-1-0239. 
The work of Kota Yoshioka was partially supported by
the Grant-in-aid for 
Scientific Research (No.\ 22340010), JSPS.

%
\section{Proof of the cone conjecture}
Let $X$ be a projective irreducible holomorphic symplectic manifold. 
Set $\Lambda:=H^{1,1}(X,\Integers)$. 
Let $\C\subset \Lambda_\RealNumbers$ be the positive cone
and $\overline{\C}$ its closure  in $\Lambda_\RealNumbers$.
A divisor class $D$ is {\em movable}, if the base locus of the linear system $\linsys{D}$ has codimension $\geq 2$
in $X$. The {\em movable cone} 
$\MV_X\subset \overline{\C}$ is the cone generated by movable divisor classes. 
Let $\MV^+_X$ be the convex hull\footnote{When $X$ is of $K3^{[n]}$-type or
of generalized Kummer type, then $\MV^+_X=\MV_X$, by 
\cite[Cor. 19]{hassett-tschinkel-moving-and-ample-cones} and \cite[Cor. 1.1]{matsushita-1310}. 
}
of $\overline{\MV}_X\cap\Lambda_{\RationalNumbers}$,
where $\overline{\MV}_X$ is the closure of the movable cone in $\Lambda_\RealNumbers$.
Let $\Bir(X)$ be the group of birational self maps of $X$.
There exists a rational polyhedral cone 
\begin{equation}
\label{eq-Pi}
\Pi\subset \MV^+_X,
\end{equation}
which is a fundamental domain for the action of 
$\Bir(X)$ on $\MV^+_X$, by \cite[Theorem 6.25]{markman-torelli}.

Assume that
Conjecture \ref{conj-self-intersection-of-extremal-classes-bounded-below} holds for $X$.
Let $\Sigma\subset \Lambda$ be the set
\[
\Sigma \ := \ \left\{f_*(e) \ : \ 
\begin{array}{l}
e \in \Nef_Y^* \ \mbox{is integral, primitive, and extremal, and} \\ 
f:Y\dashrightarrow X \ \mbox{is a birational map}
\end{array}
\right\},
\]
where $Y$ is an irreducible holomorphic symplectic manifold.
The homomorphism $f_*:H^2(Y,\Integers)\rightarrow H^2(X,\Integers)$ induced by 
a birational map is a Hodge isometry, by \cite[Prop. 1.6.2]{ogrady-weight-two}.
Hence, the set $\{(e,e) \ : \ e\in \Sigma\}$ is bounded, by our assumption. 
Note that $\Sigma$ is $\Bir(X)$-invariant, by definition.

\hide{
Let $\MV^0_X$ be the interior of the movable cone.

\begin{cor} (Follows from Proposition
\ref{prop-connected-components-of-the-birational-ample-cone})
The union $F:=\bigcup_{\lambda\in\Sigma}\MV^0_X\cap\lambda^\perp$ is a closed subset of $\MV^0_X$.
\end{cor}

\begin{proof}
The action of the signed isometry group $O^+(\Lambda)$ of the lattice $\Lambda$ on the positive cone $\C$ 
is discrete, i.e., every point $x$ has an open neighborhood intersecting the orbit 
$O^+(\Lambda)x$ in the single point $\{x\}$ \cite[Sec. 2.2]{vinberg}.
It follows that the image $Mon^2_{Bir}(X)$ of $Bir(X)$ in $O^+(\Lambda)$ acts discretely on
$\MV^0_X.$ Let $p:\MV^0_X\rightarrow \MV^0_X/Bir(X)$ be the natural map onto the topological quotient space.
Then $p$ is an open map \cite[Sec. 1.2]{vinberg}. 
The intersection $\Pi\cap \MV^0_X$ is a fundamental domain for the action of $Bir(X)$ on $\MV^0_X$.
It follows that the restriction $p:[\Pi\cap \MV^0_X]\rightarrow \MV^0_X/Bir(X)$ is an open map. 
The intersection $F\cap \Pi$ is a closed subset, since the set (\ref{eq-sigma-Pi}) is finite. Hence,
The complement $F^c$ of $F$ intersects $\Pi\cap \MV^0_X$ in an open set of the latter and maps onto an
open subset of $\MV^0_X/Bir(X)$. Now $F$ is 
$Bir(X)$-invariant and so 
\[
F^c\cap \MV^0_X \ = \ p^{-1}p[F^c\cap \MV^0_X] \ = \  p^{-1}p[F^c\cap \Pi\cap \MV^0_X].
\]
It follows that $F^c\cap \MV^0_X$ is an open subset.
\end{proof}
}
Let $\MV^0_X$ be the interior of the movable cone. 
Set $F:=\bigcup_{\lambda\in\Sigma}\MV^0_X\cap\lambda^\perp$
and $\BA_X:=\MV^0_X\setminus F$.
We will refer to $\BA_X$ as the {\em birational ample cone} in view of the following proposition.
Let $Y$ be an irreducible holomorphic symplectic manifold and $f:Y\dashrightarrow X$ a birational map.

\begin{prop}
[{\cite[Prop. 17]{hassett-tschinkel-moving-and-ample-cones} and 
\cite[Prop. 3]{BHT}}]
\label{prop-connected-components-of-the-birational-ample-cone}
The image $f_*(\Amp_Y)$ is a connected component of $\BA_X$. 
Furthermore, every connected component of $\BA_X$ is of this form. 
\end{prop}

\begin{proof}
An isomorphism $\psi:H^2(Y,\Integers)\rightarrow H^2(X,\Integers)$ is said to be a
{\em parallel transport operator}, if there exists a smooth and proper family
$\pi:\X\rightarrow B$ over an analytic space $B$, with K\"{a}hler fibers, points
$b_1, b_2\in B$, isomorphisms
$Y\cong \X_{b_1}$ and $X\cong \X_{b_2}$, and a continuous path $\gamma$ from $b_1$ to $b_2$,
such that parallel transport along $\gamma$ in the local system $R^2\pi_*\Integers$ induces the
homomorphism $\psi$.

The Hodge isometry $f_*:H^2(Y,\Integers)\rightarrow H^2(X,\Integers)$ is a
parallel transport operator, by work of Huybrechts \cite[Cor. 2.7]{huybrechts-kahler-cone} 
(see also \cite[Theorem 3.1]{markman-torelli}).
Let $Z$ be an irreducible holomorphic symplectic manifold and 
$g:Z\dashrightarrow X$
a birational map. The composition $\phi:=f_*^{-1}\circ g_*:H^2(Z,\Integers)\rightarrow H^2(Y,\Integers)$
is thus a parallel transport operator. All extremal rays of $\Nef_X^*$ are generated by classes of rational curves, by the Cone Theorem \cite[Prop. 11]{hassett-tschinkel-moving-and-ample-cones}.
Let $e\in \Nef_Z^*$ be the class of a rational curve generating  an extremal ray.
Then, either $\phi(e)$ or $-\phi(e)$ is the class of an effective $1$-cycle, by \cite[Prop. 3]{BHT}. 
We conclude that the intersection $f_*(\Amp_Y)\cap g_*(e)^\perp$ is empty. 
Hence, $f_*(\Amp_Y)$ is contained in $\BA_X$.

Clearly, $f_*(\Amp_Y)$ is an open subset of $\BA_X$ and
$f_*(\Nef_Y)$ is a closed subset of $\Lambda_\RealNumbers$. 
The cone $f_*(\Nef_Y)$ is locally rational polyhedral in a neighborhood of any class of $\MV^0_X$, 
by \cite[Prop. 17 and Cor. 19]{hassett-tschinkel-moving-and-ample-cones}. Consequently, 
we get the equality $f_*(\Amp_Y)=f_*(\Nef_Y)\cap \BA_X$, by definition of $\BA_X$. 
Hence, $f_*(\Amp_Y)$ is an open and closed subset of $\BA_X$.

The union of $f_*(\Amp_Y)$, as $f$ varies over all birational maps from all irreducible holomorphic symplectic manifolds birational to $X$, is a dense open subset of $\MV_X$, 
by   \cite[Prop. 17 and Cor. 19]{hassett-tschinkel-moving-and-ample-cones}.
Hence, every connected component
of $\BA_X$ has the form in the statement.
\end{proof}

Let $N$ be a positive integer. Let $\Pi\subset \overline{\C}$ be a rational polyhedral cone.
The following elementary statement will be proven in section \ref{sec-set-of-intersectiong-walls-is-finite}
Proposition \ref{prop-set-of-intersectiong-walls-is-finite-v2}.

\begin{prop}
\label{prop-set-of-intersectiong-walls-is-finite}
The set 
$\{\lambda\in \Lambda \ : (\lambda,\lambda)>-N \ \mbox{and} \ \lambda^\perp\cap\Pi\cap \C\neq \emptyset\}$
is finite.
\end{prop}

We conclude that the set
\begin{equation}
\label{eq-sigma-Pi}
\{\lambda\in \Sigma \ : \ \lambda^\perp\cap\Pi\cap \C\neq \emptyset\}
\end{equation}
is finite, by the above proposition and the assumption that Conjecture 
\ref{conj-self-intersection-of-extremal-classes-bounded-below} holds for $X$.

The set (\ref{eq-sigma-Pi}) divides the fundamental domain $\Pi$, given in (\ref{eq-Pi}), 
into a finite union of closed rational polyhedral subcones
\begin{equation}
\label{eq-subcone}
\Pi_i, \ \ i\in I,
\end{equation}
each with a non-empty interior.

Let $\Pi_i$ be one of the  subcones in (\ref{eq-subcone}). Let $f:Y\dashrightarrow X$ be a birational map
as above.
\begin{lem}
\label{lemma-if-translate-itersects-Nef-cone-then-it-is-contained-in-it}
If $g$ is an element of $\Bir(X)$, such that $g(\Pi_i)$ intersects the interior 
$f_*(\Amp_Y)$ of 
$f_*(\Nef_Y)$, then $g(\Pi_i)=f_*(\Nef_Y)\cap g(\Pi)$.
\end{lem}

\begin{proof}
Set $h:=g^{-1}\circ f$ and assume that the interior of $h_*(\Nef_Y)\cap\Pi_i$ is non-empty. 
We need to prove the equality $\Pi_i=h_*(\Nef_Y)\cap \Pi$. Denote by 
$\Pi_i^0$ the interior of $\Pi_i$. It suffices to prove the equality 
\[
\Pi_i^0=h_*(\Amp_Y)\cap \Pi^0.
\]
The intersection $h_*(\Amp_Y)\cap \Pi^0$ is connected, since both 
$h_*(\Amp_Y)$ and $\Pi^0$ are convex cones. Furthermore, the latter intersection is
disjoint from the hyperplane $\lambda^\perp$, for every $\lambda\in\Sigma$, 
by Proposition \ref{prop-connected-components-of-the-birational-ample-cone}. 
Hence, $h_*(\Amp_Y)\cap \Pi^0$ is contained in $\Pi_i^0$.
The inclusion $\Pi_i^0\subset h_*(\Amp_Y)\cap \Pi^0$ follows immediately from 
Proposition \ref{prop-connected-components-of-the-birational-ample-cone}. 
\end{proof}

Let
\[
I_f
\] 
be the subset of $I$, consisting of indices $i$ admitting an element
$g_i\in \Bir(X)$, such that $g_i(\Pi_i)$ is contained in $f_*(\Nef_Y)$. 
The set $I_f$ is non-empty, since $\Pi$ is a fundamental domain for the $\Bir(X)$-action on $\MV^+_X$
and the interior of $f_*(\Nef_Y)$ is contained in $\MV^+_X$. 

Let $f_j:Y_j\dashrightarrow X$ be a birational map, with $Y_j$ an irreducible holomorphic symplectic manifold,
$j=1,2$.

\begin{lem}
\label{lem}
\begin{enumerate}
\item
\label{iff}
$Y_1$ is isomorphic to $Y_2$, if and only if $I_{f_1}=I_{f_2}$.
\item
\label{partition-of-I}
If the intersection $I_{f_1}\cap I_{f_2}$ is nonempty, then $I_{f_1}=I_{f_2}$.
\end{enumerate}
\end{lem}
\begin{proof}
(\ref{iff}) 
Let $\phi:Y_1\rightarrow Y_2$ be an isomorphism. 
Let $g$ be an element of $\Bir(X)$ and $\Pi_i$ a subcone 
of $\Pi$, among those given in (\ref{eq-subcone}). 
Set $\psi:=f_2\phi f_1^{-1}$. Then
$\psi(f_1(\Nef_{Y_1}))=f_2(\Nef_{Y_2})$. 
Thus, $g(\Pi_i)$ is contained in $f_1(\Nef_{Y_1})$, 
if and only if
$(\psi g)(\Pi_i)$ is contained in $f_2(\Nef_{Y_2})$. Consequently, $\Pi_i$ belongs to $I_{f_1}$,
if and only if it belongs to $I_{f_2}$.

Assume that $I_{f_1}=I_{f_2}$. Then there exists a subcone 
$\Pi_i$ of $\Pi$, among those given in (\ref{eq-subcone}), and elements $h_j$ in $\Bir(X)$, $j=1,2$, such that $h_j(\Pi_i)$
is contained in $f_j(\Nef_{Y_j})$.
Then $f_j^{-1}h_j$ maps $\Pi_i$ into $\Nef_{Y_j}$, and so 
$f_2^{-1}h_2h_1^{-1}f_1:Y_1\rightarrow Y_2$ maps some ample class to an ample class, and is thus an isomorphism. 

(\ref{partition-of-I}) Follows from the proof of part (\ref{iff}).
\end{proof}

Given an irreducible holomorphic symplectic manifold $Y$ birational to $X$,
set
\[
I_Y:=I_f,
\]
where $f:Y\dasharrow X$ is a birational map. $I_Y$ is independent of the choice of $f$, by Lemma \ref{lem}.

\begin{cor}
\label{cor-finitely-many-isomorphism-classes}
The set ${\mathfrak B}_X$, of isomorphism classes of irreducible holomorphic symplectic manifolds 
in the birational class of $X$, is finite.
\end{cor}

\begin{proof}
The set $I$ is finite. The map $Y\mapsto I_Y$ induces a bijection between 
the set ${\mathfrak B}_X$ and subsets in the partition $I=\bigcup_{Y\in {\mathfrak B}_X} I_Y$
of $I$, by Lemma
\ref{lem}.
\end{proof}

\begin{lem}
\label{lemma-right-coset}
Given $i\in I$, the set $\{g\in \Bir(X) \ : \ g(\Pi_i)\subset \Nef_X\}$ is a left $\Aut(X)$-coset.
\end{lem}

\begin{proof}
Assume that $g(\Pi_i)$ and $h(\Pi_i)$ are both contained in $\Nef_X$ and let $\alpha$ be a class
in the interior of $\Pi_i$. Then the classes $g(\alpha)$ and $h(\alpha)$ are ample
and $gh^{-1}$ maps the ample class $h(\alpha)$ to an ample class and is thus an automorphism.
\end{proof}

Choose an element $g_i$ in the left $\Aut(X)$-coset associated to $\Pi_i$ in
Lemma \ref{lemma-right-coset}, for each $i\in I_X$.
Let $\Nef_X^+$ be the convex hull of $\Nef_X\cap\Lambda_\RationalNumbers$.

\begin{cor} 
\label{cor-nef-cone-thm}
$\Nef_X^+$ is the union of  $\Aut(X)$-translates of finitely many rational polyhedral 
subcones $g_i(\Pi_i)$, $i\in I_X$, of $\Nef_X^+$.
\end{cor}

\begin{proof}
$\Nef_X^+$ is contained in $\MV^+_X$ and the latter is a union of $\Bir(X)$-translates 
of the $\Pi_i$'s. $\Nef_X^+$ is equal to the union of the translates of the $\Pi_i$ intersecting its interior, 
by Lemma \ref{lemma-if-translate-itersects-Nef-cone-then-it-is-contained-in-it}. 
These translates are the union of the $\Aut(X)$-translates 
of $g_i(\Pi_i)$, $i\in I_X$, by Lemma \ref{lemma-right-coset}. The set $I_X$ is finite, being a subset of the finite set $I$
in Equation (\ref{eq-subcone}).
\end{proof}

Let $G$ be the image of $\Aut(X)$ in the isometry group of $\Lambda$.
Let $y$ be a rational ample class in $\Amp_X$, whose stabilizer subgroup in $G$ is trivial.
Consider the following {\em Dirichlet domain} 
\begin{equation}
\label{eq-D-y}
D_y \ := \ \{x\in\Nef_X \ : \ (x,y)\leq (x,g(y)), \ \mbox{for all} \ g\in G
\}.
\end{equation}
The following Lemma was proven by Totaro in \cite[Lemma 2.2]{totaro}.
Totaro used techniques of hyperbolic geometry. Another approach 
to the proof of the Lemma can be found in Looijenga's work \cite[Application 4.15]{looijenga}.

\begin{lem}
\label{lemma-totaro}
Suppose we are given a finite set of rational polyhedral cones in $\Nef_X^+$, such that $\Nef_X^+$
is the union of their $G$-translates. Let $y$ be a rational point 
in the interior of one of these rational polyhedral cones, whose stabilizer in $G$ is trivial.
Then the Dirichlet domain $D_y$ given above is rational polyhedral, it is contained in $\Nef_X^+$, and 
${\displaystyle \Nef_X^+=\bigcup_{g\in G}  g(D_y).}$
\end{lem}

\begin{thm} 
\label{thm-ample-cone-conj}
The Dirichlet domain $D_y$, given in Equation (\ref{eq-D-y}),  
is a  rational polyhedral cone, which is a fundamental domain for the action of $G$
on $\Nef_X^+$. In particular, $D_y$ is contained in $\Nef_X^+$. 
\end{thm} 

\begin{proof}
The statement follows from Corollary \ref{cor-nef-cone-thm}, by Lemma \ref{lemma-totaro}.
\end{proof}

\section{A rational polyhedral cone intersects only finitely many walls}
\label{sec-set-of-intersectiong-walls-is-finite}
We prove Proposition \ref{prop-set-of-intersectiong-walls-is-finite} in this section.
%
Let $\Lambda$ be a lattice of signature $(1,n-1)$. We abbreviate $(x,x)$ by $(x^2)$.
Then the cone
$$
\{x \in \Lambda_{\Bbb R} \mid (x^2 )>0 \}
$$
has two connected components.
We take $h \in \Lambda$ with $(h^2)>0$.
Then 
$$
\C^+:=\{x \in \Lambda_{\Bbb R} \mid (x^2 )>0, (x,h)>0 \}
$$
is a connected component.
For $x \in \C^+$,
we have a decomposition
$x=a h+\xi$, $(\xi,h)=0$.
\begin{NB}
$0<(x^2)=a^2(h^2)+(\xi^2)$.
Hence the sign of $a=(x,h)/(h^2)$ is constant on a connected component.
Hence $a>0$ fix the connected component. 
\end{NB}

\begin{lem}\label{lem:intersection}
For $x_1,x_2 \in \overline{\C^+}$ with $x_1 \ne 0$ and
$x_2 \ne 0$,
$(x_1,x_2)>0$ unless
$(x_1^2)=0$ and $x_2 \in {\Bbb R}x_1$.
\end{lem} 

\begin{proof}
We write $x_1=a_1 h+\xi_1$ and $x_2=a_2 h+\xi_2$,
where $a_1,a_2>0$ and 
$\xi_1, \xi_2 \in h^\perp$.
Then the Schwarz inequality imiplies
that $|(\xi_1,\xi_2)| \leq \sqrt{-(\xi_1^2)}\sqrt{-(\xi_2^2)}$.
\begin{NB}
$-(\bullet,\bullet)$ is positive definite on
$h^\perp$.
\end{NB}
Since $(x_1^2),(x_2^2) \geq 0$,
$a_1 \sqrt{(h^2)} \geq \sqrt{-(\xi_1^2)}$ and
$a_2 \sqrt{(h^2)} \geq \sqrt{-(\xi_2^2)}$.
Hence we have  
$(x_1,x_2)=a_1 a_2 (h^2)+(\xi_1,\xi_2) \geq 0$.
Moreover if the equality holds, then
$a_1 \sqrt{(h^2)}= \sqrt{-(\xi_1^2)}$,
$a_2 \sqrt{(h^2)}= \sqrt{-(\xi_2^2)}$ and
$-(\xi_1,\xi_2)=\sqrt{-(\xi_1^2)}\sqrt{-(\xi_2^2)}$.
Hence $(x_1^2)=(x_2^2)=0$. 
If $\xi_1=0$, then $(x_1^2)=0$ implies that $x_1=0$.
Hence $\xi_1 \ne 0$. We also have $\xi_2 \ne 0$.
Then
$\xi_1= y \xi_2$, $y \in {\Bbb R}_{>0}$.
Since $a_1^2 (h^2)=-y^2(\xi_2^2)$ and
$a_2^2 (h^2)=-(\xi_2^2)$, we have $a_1=y a_2$, which implies that 
$x_1=y x_2$.
\end{proof}

Assume that $x_1, x_2 \in \overline{\C^+}$ and
${\Bbb R} x_1+{\Bbb R}x_2$ is a 2-plane.
Then $((x_1+x_2)^2)>0$.

\begin{lem}\label{lem:P-finite} 
Let $P$ be a 2-plane in $\Lambda_{\Bbb R}$ defined over
${\Bbb Q}$.
If $P_{\Bbb Q}$ contains an isotropic vector $x$,
then 
$$
\{ v \in \Lambda \mid v^\perp \cap P \cap \C^+ \ne \emptyset,
(v^2)>-N  \} 
$$  
is a finite set.
\end{lem}

\begin{proof}
We may assume that $x \in \Lambda$.
If the intersection  $P \cap \C^+$ is empty, we are done. 
Assume that the intersection is non-empty.
Since 
$P_{\Bbb Q}$ is dense in $P$,
we can choose $y \in P_{\Bbb Q} \cap \C^+$.
Since $h^\perp$ is negative definite,
$(h, x) \ne 0$. We may assume that $(x,h)>0$.
Then $x \in \overline{\C^+}$.
\begin{NB}
$((x+th)^2)>0$ for $0<t$.
\end{NB}
By Lemma \ref{lem:intersection},
we have $(x,y)>0$.
Set $z:=y-\frac{(y^2)}{2(x,y)}x$. Then 
$(z,x)>0$ and $(z^2)=0$.
Replacing $z$ by $m z$, $m \in {\Bbb Z}_{>0}$,
we may assume that $z \in \Lambda$. 
An element $v$ of $\Lambda$ admits the decomposition 
$v=a x+b z+\xi$, with $\xi \in P^\perp$.
If $v^\perp \cap P \cap \C^+ \ne \emptyset$, then
$ab<0$. In that case $ab$ is bounded, $-N<ab<0$, 
since $N>-(v^2)=-ab(x,z)-(\xi^2) \geq -ab>0$. 
Now,
$a=\frac{(v,z)}{(x,z)}$ 
and $b=\frac{(v,x)}{(x,z)}$ belong to $\frac{1}{(x,z)}{\Bbb Z}$, since $v \in \Lambda$.
Therefore the choice of $a,b$ is finite. 
The element $(x,z) \xi$ belongs to the negative definite sublattice of 
$\Lambda$ orthogonal to $P$.
The choice of $\xi$ is also finite, since $N>N+ab(x,z)>-(\xi^2) \geq 0$. 
Therefore our claim holds.
\end{proof}

\begin{lem}\label{lem:P-finite2}
Assume that 
$x_1, x_2$ is a linearly independent pair in $\Lambda \cap \C^+$.
Then 
\begin{equation}
\label{eq-set-of-wall-intersecting-a-line-segment}
\{ v \in \Lambda \mid (v,s x_1+(1-s)x_2)=0, \ \mbox{for some} \ 0 \leq s \leq 1, 
(v^2)>-N \}
\end{equation} 
is a finite set.
\end{lem}

\begin{proof}
We first note that
$x_2-\frac{(x_1,x_2)}{(x_1^2)}x_1 \in x_1^\perp$
is not zero and must thus have negative self intersection.
Hence $(x_1,x_2)^2-(x_1^2)(x_2^2)>0$.
An element $v$ of $\Lambda$ admits the decomposition
$v=a x_1+b x_2+\xi$, with $\xi \in x_1^\perp \cap x_2^\perp$. 
The coefficients $a, b$ belong to 
\begin{equation}
\label{eq-a-b-in-a-discrete-set}
\frac{1}{(x_1,x_2)^2-(x_1^2)(x_2^2)}{\Bbb Z},
\end{equation}
since 
$$
\begin{pmatrix}
(v,x_1)\\
(v,x_2)
\end{pmatrix}=
\begin{pmatrix}
(x_1^2) & (x_1,x_2)\\
(x_1,x_2) & (x_2^2)
\end{pmatrix}
\begin{pmatrix}
a\\
b
\end{pmatrix},
$$

Given $s$ in the interval $0\leq s\leq 1$ we have the inequality 
$$
(x_1,sx_1+(1-s)x_2)=s(x_1^2)+(1-s)(x_1,x_2)
\geq \min\{ (x_1^2),(x_1,x_2) \}>0.
$$
The function $\lambda(s):=\frac{(x_2,s x_1+(1-s)x_2)}{(x_1,s x_1+(1-s)x_2)}$
is thus continuous on the interval $0 \leq s \leq 1$. 
We have the vanishing
\[
(\lambda(s)x_1-x_2,s x_1+(1-s)x_2)=0.
\]
Now $\lambda(s)x_1 \ne x_2$ and $s x_1+(1-s)x_2$ has positive self intersection. 
We conclude the inequality $((\lambda(s)x_1-x_2)^2)<0$.
Therefore there are positive integers $N_1, N_2$ such that
$-N_2<((\lambda(s)x_1-x_2)^2)<-\frac{1}{N_1}$.

Assume that $v$ 
belongs to the set (\ref{eq-set-of-wall-intersecting-a-line-segment}). We get
the vanishing
$$(ax_1+b x_2,s x_1+(1-s)x_2)=0,$$ which yields 
$a=-b \lambda(s)$ and
$((a x_1+b x_2)^2)=((\lambda(s)x_1-x_2)^2)(b^2)$. Furthermore, we have
$$
-N_2 b^2 \leq ((\lambda(s)x_1-x_2)^2)b^2 \leq -\frac{b^2}{N_1}.
$$
The inequalities 
$$-N<(v^2)=((a x_1+b x_2)^2)+(\xi^2) \leq ((a x_1+b x_2)^2)$$ 
yield
$N N_1>b^2$.
The choice of $b$ is thus finite, since $b$ belongs to the discrete set in 
Equation (\ref{eq-a-b-in-a-discrete-set}).
Then the choice of $a$ is also finite,
by the equality $a=-b \lambda(s)$.
Since $-(\xi^2) < N+((a x_1+b x_2)^2) \leq N$,
the choice of $\xi$ is also finite.
Therefore the claim holds.
\end{proof}

\begin{prop}
\label{prop-set-of-intersectiong-walls-is-finite-v2}
Assume that $\Pi:=\sum_{i=1}^n {\Bbb R}_{\geq 0} x_i$ is a cone
in $\overline{\C^+}$ such that $x_i \in \Lambda_{\Bbb Q}$. 
Then
$$
\{ v \in \Lambda \mid (v^2)>-N, v^\perp \cap \Pi \cap \C^+ \ne \emptyset \}
$$
is a finite set.
\end{prop}

\begin{proof}
We may assume that $\{x_1, \dots, x_n\}$ is a minimal set of generators of the cone $\Pi$. Then the 
$x_i$'s are pairwise linearly independent. 
We set
$\Pi_{ij}:={\Bbb R}_{\geq 0} x_i+{\Bbb R}_{\geq 0} x_j$.
Applying Lemma \ref{lem:P-finite} or Lemma \ref{lem:P-finite2} we conclude that
$$
V_{ij}:=\{ v \in \Lambda \mid (v^2)>-N, v^\perp \cap \Pi_{ij} 
\cap \C^+ \ne \emptyset \}
$$
is a finite set. If $n=2$ we are done. Assume that $n\geq 3$.
Assume that $v \in \Lambda$ satisfies $(v^2)>-N$ and
$v^\perp \cap \Pi \cap \C^+ \ne \emptyset$.
If the cardinality $\#(v^\perp \cap \{ x_1,...,x_n \})$ is $\geq 2$, then
$v^\perp \cap \Pi_{ij} 
\cap \C^+ \ne \emptyset$ for some $i,j$.
Hence $v \in \cup_{i,j} V_{ij}$.
Assume that $\#(v^\perp \cap \{ x_1,...,x_n \})\leq 1$. 
Set $J:=\{i \ : 1\leq i \leq n, \ \mbox{and} \ (v,x_i)\neq 0\}$. Then $\#(J)\geq 2$.
The non-emptyness of $v^\perp \cap \Pi \cap \C^+$ implies the existence of coefficients $a_i$, such that
$\sum_{i\in J} a_i (v,x_i)=0$, where
$a_i \geq 0$ for all $i\in J$, and $a_i > 0$ for some $i\in J$. 
Then
$a_i (v,x_i)$  and $a_j(v,x_j)$ have different
sign for some pair of indices $i, j \in J$. In particular,
$(v,x_i)(v,x_j)<0$. It follows that the intersection
$v^\perp \cap \Pi_{ij} 
\cap \C^+$ is non-empty.
Hence $v \in \cup_{i,j} V_{ij}$. Therefore
our claim holds.
\end{proof}


\end{document}